\documentclass[12pt]{amsart}
\usepackage{fullpage, amsfonts, amsmath, amsthm, amssymb, enumitem, comment, physics, mathtools}

\newcommand{\C}{{\mathbb C}}
\newcommand{\R}{{\mathbb R}}

\newcommand{\Z}{{\mathbb Z}}

\newcommand{\eps}{{\epsilon}}

\renewcommand{\phi}{\varphi}  

\newtheorem{theorem}{Theorem}
\newtheorem{lemma}[theorem]{Lemma}
\newtheorem{corollary}[theorem]{Corollary}

\newtheorem{conjecture}[theorem]{Conjecture}
\theoremstyle{remark}
\newtheorem{remark}[theorem]{Remark}
\theoremstyle{definition}

\begin{document}
\title{The Bombieri--Pila determinant method}
\author{Thomas F. Bloom and Jared Duker Lichtman}

\address{Mathematical Institute, University of Oxford, Oxford, OX2 6GG, UK}

\email{bloom@maths.ox.ac.uk}

\address{Department of Mathematics, Stanford University, Stanford, CA, USA}

\email{jared.d.lichtman@gmail.com}

\maketitle

\begin{abstract}
We give a concise and accessible introduction to the real-analytic determinant method for counting integral points on algebraic curves, based on the classic 1989 paper of Bombieri and Pila.
\end{abstract}

\section{Introduction}

Consider the problem of counting the number of integral points on an algebraic curve in a box; that is, solutions $(x,y)\in\{1,\ldots,N\}^2$ to the equation $F(x,y)=0$ for an irreducible polynomial $F\in \R[x,y]$. In breakthrough work in Diophantine geometry, Bombieri and Pila \cite{BP89} obtained an essentially sharp quantitative upper bound for this problem.

\begin{theorem}[Bombieri--Pila \cite{BP89}, Pila \cite{Pi96}] \label{thm:bp}
Let $F\in\mathbb{R}[x,y]$ be an irreducible polynomial of degree $d\ge2$. If $N$ is sufficiently large, depending only on $d$, then
\[\lvert \{(x,y) \in \{1,\ldots,N\}^2 :F(x,y)=0\}\rvert \le (\log N)^{O(d)}N^{\frac{1}{d}}.\]
\end{theorem}

A particularly striking feature of the bound is that the implicit constants depend only on $d=\deg F$, not on the coefficients of $F$ themselves (which may be arbitrarily large).

Bombieri and Pila's original method \cite{BP89} gave a bound of the shape $N^{1/d+o(1)}$ with a weaker explicit form for $N^{o(1)}$. Pila \cite{Pi96} refined the method to yield the stronger bound $(\log N)^{O(d)}N^{1/d}$.
It is possible that the strong bound $O(N^{1/d})$ may hold. If so, such a bound is best possible, as witnessed by the simple example $x=y^d$.

In this note we give an accessible proof of Theorem \ref{thm:bp}. We follow \cite{BP89} and \cite{Pi96}, but simplify the presentation for the sake of clarity and to bring the key features of the method to the fore. This note is essentially self-contained, appropriate for an undergraduate level reader with basic familiarity with calculus (of a single variable) and linear algebra. The only results cited without proof are B\'ezout's theorem and the implicit function theorem.

\subsection{Features and scope of the exposition}\label{subsec:scope}
Our primary aim is to present the key ideas of the Bombieri--Pila real-analytic determinant method, in the setting of curves. We will not discuss the ideas behind subsequent works, such as the $p$-adic determinant method of Heath-Brown \cite{HB02} and a global version of Salberger \cite{Sa23}. The $p$-adic method is stronger than its real-analytic predecessor in several respects, for example in its treatment of singularities. However the real-analytic determinant method does retain some advantages, for instance to count points on transcendental (higher dimensional) varieties (see for example the work of Pila and Wilkie \cite{PW06}), or for problems of a more Archimedean nature, such as counting points `near' (i.e. using the Archimedean metric) curves and higher dimensional varieties. We give an alternative application along these lines, to counting integral points on convex curves, in Section~\ref{convex}.

To appreciate the bounds under study, we stress two key features: Firstly, the determinant method obtains strong bounds when the degree $d$ is large compared to the number of variables $n$. In particular, as in \cite{BP89} and some of \cite{HB02}, the bounds improve as the degree $d$ increases. Whereas in the opposite regime, when the degree $d$ is fixed and the number of variables $n$ sufficiently large (depending on $d$), the determinant method yields weaker results as compared to the circle method, for example (as shown in spectacular fashion by Birch \cite{Bir57}). 

We give a conceptual heuristic for this (which we outline more precisely in Section \ref{sec:detmethod}): When $n=2$, the Bombieri--Pila determinant method essentially constructs a vector space of polynomials $F\in\Z[x,y]$, with dimension growing in $d$. The goal is then to find one such polynomial which satisfies certain properties, including to vanish on the integer points of our given curve. So as $d$ grows, we have increasing degrees of freedom to construct our desired polynomial. More generally, when $n\ge3$, Heath-Brown constructs a certain vector space of polynomials $F\in\Z[x_1,\ldots,x_n]$. This turns out to work well when $(n-1)/d^{1/(n-2)}$ is small, in particular when $d$ is large compared to $n$.

Secondly, the bounds under study are uniform in $F$. In particular, such uniformity is essential for applications to higher dimensional varieties, since then one may induct on the dimension  via  `slicing' arguments, applying uniform bounds at each successive dimension. In \cite{HB02}, Heath-Brown introduced a post-hoc trick to obtain uniformity, which has recently been used in other contexts (c.f. \cite{BoB20}, \cite{BoP22}). If one does not care about uniformity in $F$, then often alternative methods perform far better (for example, if the points correspond to projective points on a curve of genus $\geq 2$, then Faltings' theorem \cite{Fal83} implies that there are only finitely many integral points).

We hope this note will foster a wider understanding of the determinant method, with large potential to spur further applications. Just in 2024, the determinant method has been applied by Greenfeld, Iliopoulou, Peluse \cite{GIP24} (via \cite{CCDN20}) to bound integer distance sets, and by Browning, Lichtman, Ter\"av\"ainen \cite{BLT24} (via \cite{BP89} and \cite{HB02}) to bound the exceptional set in the $abc$ conjecture. Such recent examples highlight the timeliness of our exposition.

\subsection{Further work} 
To offer a bit of broader context, in this section we will very briefly highlight some related results involving counting integral points and the determinant method. Since our primary goal is to give a short, elementary exposition of the Bombieri--Pila method, we will not attempt by any means to give a complete survey of the literature.

Let us introduce some convenient notation: for any $n,d\geq 2$, we write
\[X_{n,d}^{\mathbb{A}}(N) := \sup_{F}\left\lvert \{ \mathbf{x}=(x_1,\ldots,x_n)\in \{1,\ldots,N\}^{n} : F(\mathbf{x})=0\}\right\rvert,\]
where the supremum ranges over all irreducible $F\in \mathbb{Q}[x_1,\ldots,x_n]$ of degree $d$.
In this notation, the estimate in Theorem \ref{thm:bp} of Bombieri--Pila gives
\[X_{2,d}^{\mathbb{A}}(N) \ll_d N^{\frac{1}{d}+o(1)}.\]
Importantly, the implied constants depend only on the degree $d$, but are independent of $N$.
Pila \cite{Pi95} extended this result to higher dimensions by a slicing argument, showing for any $n\ge2$,
\begin{align}\label{eq:Pi95}
X_{n,d}^{\mathbb{A}}(N) \ll_{n,d} N^{n-2+\frac{1}{d}+o(1)}.
\end{align}
Here the implied constants depend only on $n,d$, but not on $N$.
Again, the example $x_1=x_2^d$ shows that this exponent is the best possible.

Most of the subsequent progress has occurred in the projective setting, where the aim is to provide bounds for
\begin{align*}
X_{n,d}^{\mathbb{P}}(N) = \sup_{F}\lvert \{ \mathbf{x}=(x_0,\ldots,x_n)\in \{1,\ldots,N\}^{n+1} :  \gcd(\mathbf{x})=1, \ F(\mathbf{x})=0 \}\rvert.
\end{align*}
Here the supremum ranges over all \emph{homogeneous} irreducible $F\in \mathbb{Q}[x_0,\ldots,x_n]$ of degree $d$. Note that \eqref{eq:Pi95} immediately implies $X_{n,d}^{\mathbb{P}}(N) \ll N^{n-1+\frac{1}{d}+o(1)}$. 

Many results in the projective setting use a $p$-adic variant of the determinant method, developed by Heath-Brown \cite{HB02}, which is inspired by the real-analytic method of Bombieri--Pila presented here. As mentioned above, since our focus is on the real-analytic method we will not discuss Heath-Brown's $p$-adic determinant method (and its subsequent generalisations) here. We restrict ourselves to very briefly highlight the state of the art.

Notably, in \cite{HB02} Heath-Brown obtained a variety of estimates, including
\begin{align*}
X_{2,d}^{\mathbb{P}}(N) \ll_d N^{\frac{2}{d}+o(1)} 
\qquad \text{and}\qquad
X_{3,d}^{\mathbb{P}}(N) \ll_d N^{2+o(1)}.
\end{align*}
This result is sharp, as the example $F(x_0,x_1,x_2,x_3)=x_0^d+x_1^d-x_2^d-x_3^d$ shows. However, if we remove the `trivial' solutions, then more can be said. In particular, if we remove the points that lie on any lines on the surface $F=0$ then Heath-Brown improves this estimate to
$\ll_d N^{1+\frac{3}{\sqrt{d}}+o(1)},$
with an even stronger result if $F$ is assumed to be non-singular.

Building on the global $p$-adic determinant method of Salberger \cite{Sa23} and others, Walsh \cite{Wa15} removed the $N^{o(1)}$ factor, proving
\begin{align}\label{eq:Wa15}
X_{2,d}^{\mathbb{P}}(N) \ll_d N^{\frac{2}{d}},
\end{align}
which is sharp for $d$ fixed.
For explicit dependence on $d$, Castryck, Cluckers, Dittmann, Nguyen \cite{CCDN20} obtained
\begin{align}
X_{2,d}^{\mathbb{P}}(N) &\ll d^4 N^{\frac{2}{d}},\\
X_{2,d}^{\mathbb{A}}(N) &\ll d^3(d + \log N) N^{\frac{1}{d}} \nonumber.
\end{align}
Recently Binyamini, Cluckers, and Novikov \cite{BCN23} used real-analytic methods to prove
\begin{align}\label{eq:BCN23}
X_{2,d}^{\mathbb{P}}(N) &\ll d^2 (\log N)^{O(1)} N^{\frac{2}{d}},\\
X_{2,d}^{\mathbb{A}}(N) &\ll d^2 (\log N)^{O(1)} N^{\frac{1}{d}} \nonumber,
\end{align}
which they show is sharp in $N$ and $d$, up to the $(\log N)^{O(1)}$ factors. Notably, in \eqref{eq:BCN23} the implicit constants are absolute, and so the bound \eqref{eq:BCN23} is preferred to \eqref{eq:Wa15} when $d$ is large compared to $N$.

In higher dimensions $n\ge3$, recall \eqref{eq:Pi95} implies that $X_{n,d}^{\mathbb{P}}(N) \ll_n N^{n-1+\frac{1}{d}+o(1)}$. It is believed that, uniformly for all $d\geq 2$,
\begin{align}\label{eq:UDGC}
X_{n,d}^{\mathbb{P}}(N) \ll_{n,d} N^{n-1+o(1)}.
\end{align}
Heath-Brown's work establishes this for $n=3$, and furthermore the elementary theory of quadratic forms may establish this bound for $d=2$ and all $n$. Interestingly, Marmon \cite{Ma10} has shown that an extension of the real-analytic method of Bombieri--Pila can be used to recover the main results of Heath-Brown, rather than the $p$-adic method employed there.

Further, \eqref{eq:UDGC} is a case of the uniform dimension growth conjecture, attributed to Heath-Brown \cite{HB02}, which posits a bound of $O_{n,d}(N^{\dim X+o(1)})$ for any integral projective variety $X\subset \mathbb P^n$ of degree $d$. This was resolved for $d\ge4$ in \cite{Sa23}; also see \cite{BHBS06}. (In fact, \cite{Sa23} fully resolves the `non-uniform' dimension growth conjecture for all $n,d\ge3$, where the implied constant may depend on $F$). Again, one may ask whether the factor of $N^{o(1)}$ in \eqref{eq:UDGC} may be sharpened, or simply removed. 
This was achieved in \cite{CDHNV23}, showing for $d\ge5$,
\begin{align*}
X_{n,d}^{\mathbb{P}}(N) \ll_n d^7 N^{n-1}.
\end{align*}
Note for $d=2$, at least a factor of $(\log N)$ is required, as the example $F(x_0,x_1,x_2,x_3) = x_0x_1 - x_2x_3$ shows, see Serre \cite[p.178]{Se89}.

As mentioned in Section \ref{subsec:scope}, determinant methods work best when $d$ is larger, and by their nature are less suited to smaller $d$ as compared to other approaches \cite{Bir57}. In particular, when $d=3$ the best known bound is $X_{n,3}^{\mathbb{P}}(N) \ll_n N^{n-1+1/7+o(1)}$ from \cite{Sa15}. As such, it is a key question for contemporary analytic number theory to remove the factor $N^{1/7}$, so to obtain the conjectured bound \eqref{eq:UDGC} when $d=3$.
All these results in turn are (type I) cases of Serre's conjecture for thin sets, see \cite{Se89,Se92} for further details. Finally, we mention that the determinant method may be viewed as an example of `polynomial methods.' See \cite{Hau17} for an introductory survey of this much broader topic, which has seen wide applications.

\subsection*{Acknowledgements}

We would like to thank Tim Browning and Roger Heath-Brown for useful conversations, as well as Raf Cluckers and the anonymous referees for helpful feedback. The first author is supported by a Royal Society University Research Fellowship. The second author is supported by an NSF Mathematical Sciences Postdoctoral Research Fellowship.

\section{Covering integer points by curves}
In this section we present the key ingredient of the proof of Theorem~\ref{thm:bp}. The main idea is that, given some smooth function $f$ which has rapidly decaying derivatives, we can cover the integer points on the graph $(x,f(x))$ by a small number of curves drawn from a certain specified set. This quickly leads to efficient upper bounds for the number of such points via an application of B\'ezout's theorem.

Let $\mathcal{M}\subseteq \mathbb{R}[x,y]$ be a finite set of monomials. We write $\langle \mathcal{M}\rangle$ for the span of $\mathcal{M}$, that is, those polynomials in $\mathbb{R}[x,y]$ whose monomials are all in $\mathcal{M}$. We will sometimes abuse notation and write $\mathbf{j}=(j_1,j_2)\in \mathcal{M}$ to mean $x^{j_1}y^{j_2}\in \mathcal{M}$. Let 
\begin{align}\label{def:pq}
p=p_{\mathcal{M}}:=\sum_{x^{j_1}y^{j_2}\in \mathcal{M}}j_1\qquad\text{ and }\qquad
q=q_{\mathcal{M}}:=\sum_{x^{j_1}y^{j_2}\in \mathcal{M}}j_2.
\end{align}
The main example to keep in mind is the case when $\mathcal M = \{x^{j_1}y^{j_2} : j_1+j_2 \le d\}$, the set of all monomials of degree at most $d$. In this case we have
$D = |\mathcal M| 
= \frac{1}{2}(d+1)(d+2)$
and
$p=q 
= \frac{1}{3}dD$, and $\langle \mathcal{M}\rangle$ is simply the set of polynomials with degree at most $d$. The need to work with the more general situation is because we will apply B\'ezout's theorem to count integer points on some curve $F(x,y)=0$, and therefore need to make sure that we are not constructing some $G\in \mathbb{R}[x,y]$ with $F\mid G$. The reader may like to look ahead to Section~\ref{sec-app} to see how to choose $\mathcal{M}$ to avoid this, but on first reading they should just take $\mathcal{M}$ to be all monomials of degree at most $d$.

The driving force of the Bombieri--Pila determinant method is the following lemma, which states that if a function $f$ has rapidly decaying derivatives then the integer points on any sufficiently short segment of the graph $(x,f(x))$ can be covered by a single curve in $\langle \mathcal{M}\rangle$. If one does not care about the quantitative aspects, some result like this is trivial, since any $D-1$ points are contained in a curve in $\langle \mathcal{M}\rangle$ by linear algebra. For sufficiently smooth functions, however, the following bound is far superior.
\begin{lemma}[Bombieri--Pila]\label{lem:ptspacing}
Let $\mathcal{M}$ be a finite set of monomials of size $D=\lvert \mathcal{M}\rvert$ and let $p,q$ be defined as in \eqref{def:pq}. Let $I\subseteq [0,N]$ be a closed interval and $f\in C^{D-1}(I)$. Suppose that $X>0$ and $\delta\geq 1/N$ are such that for all $0\leq i<D$ and $x\in I$
\[\left\lvert \frac{f^{(i)}(x)}{i!}\right\rvert\leq X\delta^i.\]
If
\begin{align*}
\lvert I\rvert < \tfrac{1}{4}\delta^{-1}\big((2N)^p(DX)^q\big)^{-1/\binom{D}{2}}
\end{align*}
then $\{(x,f(x)) : x\in I\}\cap \Z^2$ is contained in some curve in $\langle \mathcal{M}\rangle$ (that is, the zero set of some $F\in \langle \mathcal{M}\rangle$). 
\end{lemma}

For applications, the following lemma is often more convenient to apply. It covers longer pieces of the graph by (not too many) curves in $\langle\mathcal{M}\rangle$. The proof is a simple greedy application of Lemma~\ref{lem:ptspacing}.
\begin{lemma}[Bombieri--Pila]\label{lem:bp}
Let $\mathcal{M}$ be a finite set of monomials of size $D=\lvert \mathcal{M}\rvert$ and let $p,q$ be defined as in \eqref{def:pq}. Let $I\subseteq [0,N]$ be a closed interval and $f\in C^{D-1}(I)$. Suppose that $X>0$ and $\delta\geq 1/N$ are such that for all $0\leq i<D$ and $x\in I$
\[\left\lvert \frac{f^{(i)}(x)}{i!}\right\rvert\leq X \delta^i.\]
The integer points  $\{(x,f(x)) : x\in I\}\cap \Z^2$ are contained in the union of at most
\begin{align*}
4\delta\lvert I\rvert\big((2N)^p(DX)^q\big)^{1/\binom{D}{2}}+1
\end{align*}
many curves in $\langle \mathcal{M}\rangle$. 
\end{lemma}

Note that we can trivially cover such integer points by at most $\lvert I\rvert+1$ many curves, and so this lemma gives a significant saving over this trivial upper bound roughly when $\delta\ll \left(N^pD^qX^q\right)^{-1/\binom{D}{2}}$.

\begin{proof}
Let $\{(x,f(x)) : x\in I\}\cap \Z^2=\{z_1,\ldots,z_t\}$, say, arranged in increasing order of their $x$-coordinates. 
Define a sequence of integers $n_0,n_1,n_2,\ldots$ by $n_0=0$ and recursively let $n_\ell$ denote the largest index for which the points
\[\{z_i : n_{\ell-1} \le i < n_{\ell}\}\]
are contained in a single curve in $\langle \mathcal{M}\rangle$. Suppose the sequence $n_0,n_1,\ldots,n_m$ terminates after $m+1$ elements. For $0\leq \ell<m$ the set $\{z_{n_{\ell}},\ldots,z_{n_{\ell+1}}\}$ is not contained in a single curve in $\langle \mathcal{M}\rangle$, and so Lemma \ref{lem:ptspacing} implies that the length of the interval $[x_{n_\ell},\,x_{n_{\ell+1}}]$ is 
\begin{align*}
x_{n_{\ell+1}} - x_{n_\ell} \ge \tfrac{1}{4}\delta^{-1}\left((2N)^p(DX)^q\right)^{-1/\binom{D}{2}}.
\end{align*}
On the other hand, since $x_{n_0},\ldots, x_{n_m}\in I$, we have
\begin{align*}
|I| & \ge x_{n_m} - x_{n_0} = \sum_{0\leq \ell<m}(x_{n_{\ell+1}} - x_{n_\ell})
\ge \frac{m}{4}\delta^{-1}\left((2N)^p(DX)^q\right)^{-1/\binom{D}{2}}.
\end{align*}
Rearranging gives the desired bound on $m+1$, the number of curves in $\mathcal M$.
\end{proof}

\section{The determinant method} \label{sec:detmethod}
In this section we will prove Lemma~\ref{lem:ptspacing}, and explain the main idea behind the (real) determinant method. The basic structure of the proof is as follows:
\begin{enumerate}
\item[(0)] We would like to find some polynomial $F\in \langle \mathcal{M}\rangle$ which vanishes on all points in $S=\{(x,f(x)): x\in I\}\cap \mathbb{Z}^2$.
\item Consider the matrix $A$ with entries $z^{\bf j}=(x^{j_1},f(x)^{j_2})$, for $z\in S$ and ${\bf j}\in \mathcal{M}$. If $A$ has rank $<\lvert S\rvert$, then there is some linear dependency between the rows of $A$, which will mean some $F\in \langle \mathcal{M}\rangle$ vanishes on $S$, as desired (see Lemma~\ref{lem:rankPij}). 
\item For sake of contradiction, suppose this does not happen. Then there is some non-singular $\lvert S\rvert\times \lvert S\rvert$ sub-matrix $A'$. Importantly, this sub-matrix $A'$ has integer values, and hence its determinant will be at least $1$ in absolute value.
\item We force a contradiction by proving $|\det A'|<1$ directly. To do this, we control $\det(x^{j_1}f(x)^{j_2})$ by a determinant involving the derivatives of $f$ which we then bound trivially, using the Leibniz determinant formula and the assumption that the derivatives of $f$ decay rapidly. This is shown in Lemma~\ref{prop:detdet}.
\end{enumerate}

We begin with step (1) above, converting the problem into one concerning a matrix of monomials. Note that, by interpolation, any $D-1$ points in the plane lie on a common curve in $\langle \mathcal{M}\rangle$. The key insight driving the Bombieri--Pila method (and indeed most instances of the so-called polynomial method) is that this common curve can cover even more points, assuming the rank of the associated monomial matrix is not maximal.
\begin{lemma}\label{lem:rankPij}
Let $\mathcal{M}\subset \mathbb{R}[x,y]$ be some finite set of monomials. For any $z_{1},\ldots, z_{t}\in \mathbb{R}^2$ if
\begin{align*}
\rank(z_i^\mathbf{j})_{\substack{i\le t\\\mathbf{j}\in \mathcal{M}}} < |\mathcal{M}|,
\end{align*}
then $z_1,\ldots,z_t$ are contained in a curve in $\langle \mathcal{M}\rangle$. 
\end{lemma}
\begin{remark}
We shall not need it, but the converse of Lemma \ref{lem:rankPij} also holds: If $z_1,\ldots,z_t$ are contained in a curve in $\langle \mathcal{M}\rangle$, then the rank is less than $|\mathcal{M}|$.
\end{remark}
\begin{proof}
Let $\mathcal N\subseteq \mathcal M$ and $S\subseteq \{1,\ldots,t\}$ be such that $z_S^{\mathcal N}:=(z_i^\mathbf{j})_{\substack{i\in S\\\mathbf{j}\in \mathcal{N}}}$ is an $r\times r$ minor of maximal rank. By assumption $r<|\mathcal{M}|$, and so there exists $\mathbf{k}\in \mathcal{M}\setminus \mathcal{N}$.

Consider $f\in\mathbb{R}[x,y]$ given by 
\begin{align}\label{eq:fdetminors}
f(x,y) := \det\mqty((x,y)^{\mathcal{N}} & (x,y)^{\mathbf{k}} \\ z_S^{\mathcal{N}} & z_S^{\mathbf{k}}) = \sum_{\mathbf{j}\in \mathcal{N}\cup\{\mathbf{k}\}}\left( \epsilon_\mathbf{j}\det(z_{i}^{\mathbf{l}})_{\substack{i\in S\\\mathbf{l}\in \mathcal{N}\cup\{\mathbf{k}\}\setminus\{\mathbf{j}\}}}\right) x^{j_1}y^{j_2},
\end{align}
for some $\epsilon_\mathbf{j}\in\{\pm1\}$. Here in the second equality we have used the definition of determinant. From the right-hand side of \eqref{eq:fdetminors}, we see $f(x,y)$ defines a curve in $\langle \mathcal{M}\rangle$. (Note that in fact this curve is in the $\mathbb{Z}$-linear span of $\mathcal{M}$ if $z_1,\ldots,z_t\in \mathbb{Z}^2$.) It remains to note that $f(z_i)=0$ for all $1\leq i\leq t$. Indeed, if $i\in S$ the matrix in \eqref{eq:fdetminors} above has repeated rows, while if $i\notin S$ the determinant $f(x,y)$ is zero by maximality of $r$.
\end{proof}

To profitably apply Lemma~\ref{lem:rankPij}, we will need to be able to bound the determinant in a non-trivial way. This is accomplished, in the Bombieri--Pila method, via the following lemma.

\begin{lemma}\label{prop:detdet}
Let $n\geq 1$ and $I\subset \R$ be some closed interval with $x_1,\ldots,x_{n}\in I$. Suppose $f_1,\ldots,f_{n}\in C^{n-1}(I)$ and $A_{ij}\geq 0$ are such that for $1\leq i,j\leq n$ and all $x \in I$, 
\[\left\lvert \frac{f_j^{(i-1)}(x)}{(i-1)!}\right\rvert\leq A_{ij}.\]
Then we have
\[
\left\lvert\det(f_j(x_i))\right\rvert\le  \bigg(\prod_{i>j}\abs{x_i-x_j}\bigg)\sum_{\sigma \in S_n}\prod_{1\leq i\leq n}A_{i\sigma(i)},\]
where $S_n$ is the standard group of bijections $\sigma:\{1,\ldots,n\}\to\{1,\ldots,n\}$.
\end{lemma}
\begin{proof}
We first claim that for every $1\leq i,j\leq n$ there exists some $\xi_{ij}\in I$ such that
\begin{align}\label{eq:fjinterpolate}
\frac{f_j^{(i-1)}(\xi_{ij})}{(i-1)!}
&=\sum_{\ell\leq i} f_j(x_\ell) \prod_{\substack{m\leq i\\ m\neq \ell}}\frac{1}{x_\ell-x_m}. 
\end{align}
This fact is a consequence of Lagrange interpolation: consider the following polynomial, of degree $<i$, 
\[g_{ij}(\xi) = \sum_{\ell\leq i}f_j(x_\ell) \prod_{\substack{m\leq i\\ m\neq \ell}}\frac{\xi-x_m}{x_\ell-x_m}.\]
By construction, $g_{ij}$ agrees with $f_j$ at $i$ many points (namely $f_j(x_k)=g_{ij}(x_k)$ for $k\leq i$). That is, the function $f_{j}-g_{ij}$ has $i$ many zeros in $I$, and hence by repeated applications of Rolle's theorem, the $(i-1)$-fold derivative will have some zero in $I$. That is, there exists some $\xi_{ij}\in I$ such that $f_j^{(i-1)}(\xi_{ij}) = g_{ij}^{(i-1)}(\xi_{ij})$. Hence \eqref{eq:fjinterpolate} follows, since the right-hand side is the precisely the derivative $g_{ij}^{(i-1)}$ (which is constant, since $\deg g_{ij}<i$).

Let $G_{\ell}(x_1,\ldots,x_i)=(i-1)!\prod_{\substack{m\leq i\\ m\neq \ell}}\frac{1}{x_\ell-x_m}$. By \eqref{eq:fjinterpolate} it follows that
\begin{align}\label{eq:detfG}
\det\big(f_j^{(i-1)}(\xi_{ij})\big)
=\det\left(\sum_{\ell\leq i}f_j(x_\ell)G_\ell(x_1,\ldots,x_i)\right)=\left(\prod_{1\leq i\leq n}G_{i}(x_1,\ldots,x_i)\right)\det(f_j(x_i)).
\end{align}
Here we used the general fact that $\det( \sum_{\ell \leq i}a_{j\ell}b_{\ell i})=\left(\prod_{i}b_{ii}\right)\det(a_{ji})$, since the matrix on the left can be factored as $(a_{j\ell})$ times the lower triangular matrix $(b_{\ell i})_{\ell\leq i}$, the determinant of which is equal to the product of its diagonal entries. 

Therefore, since $G_{i}(x_1,\ldots,x_i)^{-1}=\frac{1}{(i-1)!}\prod_{m<i}(x_i-x_m)$, the identity \eqref{eq:detfG} implies
\begin{align*}
\left\lvert\det(f_j(x_i))\right\rvert
& \le \left(\prod_{1\leq i\leq n}\lvert G_{i}(x_1,\ldots,x_i)\rvert\right)^{-1} \lvert\det\big(f_j^{(i-1)}(\xi_{ij})\big)\rvert\\
&\le \left(\prod_{i>j}\abs{x_i-x_j}\right)\sum_{\sigma\in S_n}\prod_{i\le n}\frac{\lvert f^{(i-1)}_{\sigma(i)}(\xi_{i\sigma(i)})\rvert}{(i-1)!} \\
& \le  \left(\prod_{i> j}\abs{x_i-x_j}\right)\sum_{\sigma \in S_n}\prod_{1\leq i\leq n}A_{i\sigma(i)}
\end{align*}
by the Leibniz formula and assumed bound on the derivatives. This completes the proof.
\end{proof}

We will apply the previous lemma with $f_\mathbf{j}(x)=x^{j_1}f(x)^{j_2}$, and so it will be necessary to bound the derivatives of such functions. Note that this lemma is the only place we use the assumption that $I\subseteq [0,N]$. 
\begin{lemma}\label{prop:normmult}
Let $k\ge1$, $I\subseteq [0,N]$ a closed interval, and $f\in C^k(I)$. Suppose that $X>0$ and $\delta\geq 1/N$ are such that, for all $0\leq i\leq k$ and $x\in I$,
\[\left\lvert \frac{f^{(i)}(x)}{i!}\right\rvert\leq X\delta^i.\]
For any integer pair $\mathbf{j}=(j_1,j_2)\in \Z_{\ge0}^2$, the function $f_\mathbf{j}(x)=x^{j_1}f(x)^{j_2}$ satisfies
\[\left\lvert \frac{f_\mathbf{j}^{(i-1)}(x)}{(i-1)!}\right\rvert \leq (2N)^{j_1}(iX)^{j_2}\delta^{i-1} \]
for all $1\leq i\leq k$ and $x \in I$.
\end{lemma}
\begin{proof}
For any $1\leq i\le k$, the product rule gives
\begin{align*}
\frac{f_\mathbf{j}^{(i-1)}(x)}{(i-1)!}=\sum_{i_0+i_1+\cdots+i_{j_2}=i-1}\binom{j_1}{i_0}x^{j_1-i_0}\frac{f^{(i_1)}(x)}{i_1!}\cdots \frac{f^{(i_{j_2})}(x)}{i_{j_2}!}.
\end{align*}
Using the assumed derivative bounds on $f^{(i)}$, since $\binom{j_1}{i_0}\leq 2^{j_1}$ and $\lvert x\rvert \leq N$ for $x\in I$, 
\begin{align*}
\left\lvert\frac{f_\mathbf{j}^{(i-1)}(x)}{(i-1)!}\right\rvert 
& \le \sum_{i_0+i_1+\cdots+i_{j_2}=i-1}\binom{j_1}{i_0} |x|^{j_1-i_0} (X\delta^{i_1})\cdots(X\delta^{i_{j_2}}) \\
& \leq  \sum_{i_0+i_1+\cdots+i_{j_2}=i-1} 2^{j_1}\, N^{j_1-i_0} X^{j_2} \delta^{i-i_0-1}\\
& \leq  i^{j_2} 2^{j_1}\, N^{j_1} X^{j_2} \delta^{i-1},
\end{align*}
using $\delta N\geq 1$, and bounding the number of partitions of $(i-1)$ into $m$ parts by $i^{m-1}$.
\end{proof}

We now have enough tools to prove Lemma~\ref{lem:ptspacing}.
\begin{proof}[Proof of Lemma~\ref{lem:ptspacing}]
Let $\{(x,f(x)) : x\in I\}\cap \Z^2=\{z_1,\ldots,z_t\}$, say. By Lemma \ref{lem:rankPij}, it suffices to show that the matrix $M=(z_i^\mathbf{j})_{\mathbf{j}\in \mathcal{M}}^{1\le i\le t}$ has rank $<D$. If not, then $M$ must have maximal rank $D$, and hence there is a subset of $D$ indices $S\subseteq \{1,\ldots,t\}$ such that
\begin{align}
\Delta := \det(z_i^\mathbf{j})_{\substack{i\in S\\\mathbf{j}\in \mathcal{M}}} \neq 0.
\end{align}
Since each $z_i\in \mathbb{Z}^2$ we must have $\Delta\in \Z$, and so in particular $|\Delta|\ge1$. Relabelling if necessary, we can assume that $S=\{1,\ldots,D\}$. 

Let $f_\mathbf{j}(x)=x^{j_1}f(x)^{j_2}$, so that $f_\mathbf{j}(x_i)=z_i^\mathbf{j}$. By Lemma~\ref{prop:normmult} and the assumption on the derivatives of $f$, we have, for all $1\leq i,j\leq D$ and $x\in I$,
\[\left\lvert \frac{f_\mathbf{j}^{(i-1)}(x)}{(i-1)!}\right\rvert\leq (2N)^{j_1}(DX)^{j_2}\delta^{i-1}.\]
Thus by Lemma~\ref{prop:detdet} with $f_\mathbf{j}(x) := x^{j_1}f(x)^{j_2}$ for $\mathbf{j}=(j_1,j_2)\in \mathcal M$ and $A_{i\mathbf{j}}=(2N)^{j_1}(DX)^{j_2}\delta^{i-1}$,
\begin{align}\label{eq:1Delta}
1\le |\Delta| = \lvert\det(f_\mathbf{j}(x_i))_{\substack{i\in S\\\mathbf{j}\in \mathcal{M}}}\rvert
&\le  \left(\prod_{i> j}\abs{x_i-x_j}\right)\sum_{\sigma} \prod_{i\in S}A_{i\sigma(i)} \nonumber\\
&\le  \prod_{i> j}|I|\sum_{\sigma} \prod_{i\in S}(2N)^{\sigma(i)_1}(DX)^{\sigma(i)_2}\,\delta^{i-1} \nonumber\\
& \le  \lvert I\rvert^{\binom{D}{2}}\, D!\,(2N)^p(DX)^q\, \delta^{\binom{D}{2}},
\end{align}
recalling the quantities $p$ and $q$ from \eqref{def:pq}. Here $\sigma$ ranges over all bijections from $S\to \mathcal{M}$. Using the crude bounds $D!\leq D^D$ and $D^{\frac{2}{D-1}}\leq 4$ for all $D\ge2$, isolating $|I|$ in \eqref{eq:1Delta} above gives
\begin{align*}
\lvert I\rvert & \ge D^{-D/\binom{D}{2}}\,[(2N)^p(DX)^q]^{-1/\binom{D}{2}}\, \delta^{-1}\\
&= D^{-\frac{2}{D-1}}\,[(2N)^p(DX)^q]^{-1/\binom{D}{2}}\, \delta^{-1}\\
& \ge\tfrac{1}{4}\delta^{-1}[(2N)^p(DX)^q]^{-1/\binom{D}{2}},
\end{align*}
which contradicts our assumption. Hence rank$(M)<D$, and thus an application of Lemma \ref{lem:rankPij} covers $\{z_1,\ldots,z_t\}$ with a single curve in $\mathcal M$, as desired.
\end{proof}

\section{Application of the key lemma}\label{sec-app}
We now show how to use the determinant method (more precisely, its consequence in the form of Lemma~\ref{lem:bp}) to bound the number of integral points on curves. This uses B\'ezout's theorem.
\begin{theorem}[B\'ezout's theorem]
Let $F,G\in \mathbb{R}[x,y]$ be non-constant polynomials with no common divisor in $\mathbb{R}[x,y]$. There are at most $\deg F\cdot \deg G$ many points $(x,y)\in \mathbb{R}^2$, counted with multiplicity, such that $F(x,y)=G(x,y)=0$.
\end{theorem}
We will only require B\'ezout's theorem for irreducible $F$, where it takes the following form. Recall that $F\in \mathbb{R}[x,y]$ is irreducible if it cannot be factored into the product of two non-constant polynomials in $\mathbb{R}[x,y]$.
\begin{corollary}\label{cor:Bzout}
Let $F\in \R[x,y]$ be an irreducible polynomial and $G\in \R[x,y]$ such that $F\nmid G$. There are at most $\deg F\cdot \deg G$ many points $(x,y)\in \R^2$ such that $F(x,y)=G(x,y)=0$.
\end{corollary}

We use B\'ezout's theorem as our fundamental tool to count integer points on arbitrary curves, by covering such points with other curves of bounded degree.

\begin{lemma}\label{lem-intbound1}
Let $\ell \geq d\geq 2$. Let $I\subseteq [0,N]$ be a closed interval, and $f\in C^{\infty}(I)$ with $F(x,f)=0$ for some irreducible polynomial $F\in \R[x,y]$ of degree $d\geq 2$. Suppose that $N>0$ and $\delta\geq 1/\lvert I\rvert$ satisfy
\[\left\lvert \frac{f^{(i)}(x)}{i!}\right\rvert\leq N\delta^i\qquad {\rm for \ all }\quad x\in I,\]
for every $0\leq i<D=d(\ell-d+1)$. Then we have
\begin{align*}
\left\lvert\big\{(x,f(x)) : x\in I\big\}\cap \Z^2\right\rvert \ll(d\ell)^2 N^{\frac{1}{d}+O(\frac{1}{\ell})} \delta\lvert I\rvert.
\end{align*}
\end{lemma}

In particular if the $i$th derivative of $f$ decays like $\abs{I}^{-i}N^{1+o(1)}$ (as we will shortly show can be arranged in practice) then Lemma \ref{lem-intbound1} gives the desired upper bound of $N^{\frac{1}{d}+o(1)}$, taking $\ell=O(\log N)$.
\begin{proof}
Let $i_F=i$ be the maximal index such that $x^{d-i_F}y^{i_F}$ is a monomial in $F$. Define
\[\mathcal{M} = \{ x^{j_1}y^{j_2} : d\leq j_1+j_2\leq \ell\textrm{ and }x^{d-i_F}y^{i_F}\nmid x^{j_1}y^{j_2}\}.\]
Next for each integer $h\in[d,\ell]$, we observe
\begin{align*}
|\{(j_1,j_2)\in \mathcal{M} : h=j_1+j_2\}| =d.
\end{align*}
Indeed, such $\mathbf{j}$ are of the form $\mathbf{j}=(j_1,h-j_1)$ for $j_1\in [0,h]$. The condition $x^{d-i_F}y^{i_F}\nmid x^{j_1}y^{j_2}$ implies $d-i_F > j_1$ or $i_F>h-j_1$. The first case gives $j_1\in [0,d-i_F)$, and the second case gives $j_1\in (h-i_F,h]$. Since $h\geq d$ these combine for $d-i_F + i_F = d$ choices of $j_1$. Hence $|\{(j_1,j_2)\in \mathcal{M} : h=j_1+j_2\}| =d$.

Thus we obtain
\begin{equation}
D =\lvert \mathcal{M}\rvert = \sum_{d\le h\le \ell} d  =d(\ell-d+1) = d\ell+O(d^2) \label{eq:Dcompute}
\end{equation}
and
\begin{align}
p+q &=\sum_{d\le h\le \ell} hd=\frac{d}{2}(\ell(\ell+1)-d(d-1)) \label{eq:p+qcompute} \\
&\leq\frac{\ell^2d}{2}+O(\ell d)=\frac{D^2}{2d}+O(D^2/\ell).\nonumber
\end{align}
for $p,q$ as in \eqref{def:pq}. In particular, 
\begin{align}\label{eq:pDchoose2}
\frac{p+q}{\binom{D}{2}} =\frac{2(p+q)}{D(D-1)} &\le \frac{\frac{D}{d}+O(D/\ell)}{D-1}
\le \frac{1}{d}+ O(1/\ell ).
\end{align}

Now we claim that $F\nmid G$ for all $G=G(x,y)\in\langle \mathcal{M}\rangle$. Indeed, if not then $G=FH$, where $H$ has degree $d'$, say. Let $i_H$ be maximal such that $x^{d'-i_H}y^{i_H}$ is a monomial in $H$. Then $x^{d+d'-i_F-i_H}y^{i_F+i_H}$ is a monomial in $G$ (note that by maximality of $i_F$ and $i_H$ no other monomial of degree $d+d'$ can cancel it) which is divisible by $x^{d-i_F}y^{i_F}$, and hence is not in $\mathcal{M}$, a contradiction. Hence $F\nmid G$.

Therefore, by B\'ezout's theorem in the form of Corollary~\ref{cor:Bzout}, each curve $G(x,y)\in\langle \mathcal{M}\rangle$ can contain at most $d\ell$ many points in the graph $\Gamma=\{(x,f(x)) : x\in\mathbb{R}\}$. We may take
\begin{align*}
\le 4\delta|I| \big((2N)^{p}(DN)^q\big)^{1/\binom{D}{2}}+1
\end{align*}
many curves $G(x,y)\in\langle \mathcal{M}\rangle$ to cover $\Gamma\cap\Z^2$, by the key Lemma \ref{lem:bp}.
 Hence we conclude
\begin{align*}
|\Gamma\cap \Z^2|
&\ll d\ell\Big(\delta\lvert I\rvert\big(N^{p+q}D^q\big)^{1/\binom{D}{2}}+1\Big) \\
&\ll d\ell\left( \delta\lvert I\rvert\big(ND\big)^{\frac{1}{d}+O(\frac{1}{\ell})}+1\right) \\
& \ll (d\ell)^2 \delta\lvert I\rvert N^{\frac{1}{d}+O(\frac{1}{\ell})}
\end{align*}
using $1/\delta \le |I|$ and \eqref{eq:pDchoose2}.
\end{proof}

To deduce Theorem~\ref{thm:bp} from Lemma~\ref{lem-intbound1} it remains to divide the curve inside $\{1,\ldots,N\}^2$ into a small number of pieces which locally look like a graph $(x,f(x))$ for some $f$ with sufficiently rapidly decaying derivatives. This can be done in several ways; in particular Pila and Wilkie \cite{PW06} have shown that this can be efficiently done using a lemma of Gromov \cite{Gr87} and Yomdin \cite{Yo87}. This method was used (and simplified) by Marmon \cite{Ma10} in his recent extension of the real-analytic method. In this note we will follow Bombieri and Pila and use a less efficient but more elementary approach, greedily dividing the curve into pieces with small derivatives aside from a small number of (very short) exceptional intervals on which the derivative is too large.

The following technical lemma is preparation for such a division. 
\begin{lemma}\label{lem:5}
Suppose $F(x,y)\in\R[x,y]$ is an irreducible polynomial of degree $d$. Let $I$ be an interval and $f\in C^\infty(I)$ satisfy $F(x,f(x))=0$. Then for any $k\ge1$ and $c\in\R\backslash\{0\}$,
\begin{align*}
\big|\big\{ x\in I : f^{(k)}(x) = c \big\}\big|\ll kd^2.
\end{align*}
\end{lemma}
\begin{proof}
Suppose first that $f$ is a polynomial, necessarily of degree $\leq d$. Then $f^{(k)}$ is a polynomial of degree $\leq d-k$, whence $f^{(k)}(x)=c$ has at most $d-k$ many solutions.

We may now assume that $f$ is not a polynomial. We claim that for each $1\leq k\leq d$, there is a polynomial $H_k\in \R[x,y]$ of degree at most $d_k:=(k-1)(2d-3)+d-1$ such that, for all $x\in I$,
\begin{equation}\label{eq:indH}
H_k(x,f)+ F_y(x,f)^{2k-1} f^{(k)}(x)=0,
\end{equation}
where we write $F_y(x,y) = \frac{\partial}{\partial y}F(x,y)$, which is a polynomial of degree $\leq d-1$.

We prove \eqref{eq:indH} by induction on $k$. For $k=1$, differentiating the identity $F(x,f)=0$ with respect to $x$ gives $F_x(x,f)+ F_y(x,f)f'(x)=0$, with $H_1 = F_x$ of degree $d-1$.
For the inductive step, assuming \eqref{eq:indH} holds for $k\ge1$, differentiating with respect to $x$ gives
\begin{align*}
0 & = (H_{k})_x(x,f) + (H_{k})_y(x,f)f'(x)+ F_y(x,f)^{2k-1}\,f^{(k+1)}(x)\\
& +(2k-1) F_y(x,f)^{2k-2}\,f^{(k)}(x)\big(F_{xy}(x,f) + F_{yy}(x,f)f'(x)\big).
\end{align*}
We then multiply this equation by $F_y^2$, giving
\begin{align*}
0 & = F_y^2\big((H_{k})_x + (H_{k})_yf'\big)+ F_y^{2k+1}\,f^{(k+1)}
+(2k-1) F_y^{2k}\,f^{(k)}\big( F_{xy} + F_{yy}\,f'\big).
\end{align*}
Eliminating $f^{(k)}$ and $f'$, using $ F_y^{2k-1} \, f^{(k)}=-H_k$ from \eqref{eq:indH} and $F_y\,f'=-F_x$,
\begin{align*}
0 & = F_y^2(H_{k})_x - F_yF_x (H_{k})_y
-(2k-1) H_k\big( F_y F_{xy} - F_x F_{yy}\big) + F_y^{2k+1}\,f^{(k+1)}\\
& =: \ H_{k+1}(x,f) +  F_y^{2k+1} f^{(k+1)},
\end{align*}
where $H_{k+1}$ has degree at most $d_k+2d-3=d_{k+1}$. This completes the induction for \eqref{eq:indH}.

It follows that, for any constant $c\in \R$, the solutions $x\in I$ to $f^{(k)}(x)=c$ must satisfy $R_c(x,f(x))=0$, where $R_c(x,y)=H_k(x,y)+ F_y(x,y)^{2k-1}c$. Since they must also satisfy $F(x,f(x))=0$, it suffices to show $F\nmid R_c$, whence B\'ezout's theorem in the form of Corollary~\ref{cor:Bzout} bounds the number of common points $x$ by $(\deg F)(\deg R_c)\le d(2kd)\ll kd^2$, as desired.

It remains to derive a contradiction if $F\mid R_c$. Note that this means that $F(x,y)=0$ implies $R_c(x,y)=0$ for any $x,y\in \R$. In particular, our assumption $F(x,f(x))=0$ for all $x\in I$ implies
\begin{align}\label{eq:Rc=0}
0=R_c(x,f(x)) &= H_k(x,f(x))+ F_y\big(x,f(x)\big)^{2k-1}c \nonumber\\
&= F_y\big(x,f(x)\big)^{2k-1}\,\big(c-f^{(k)}(x)\big)
\end{align}
for all $x\in I$, recalling \eqref{eq:indH}. Since $F_y\big(x,f(x)\big)$ is a polynomial of degree $\leq d-1$, by Corollary~\ref{cor:Bzout} there are $O(d^2)$ (in particular finitely many) $x\in I$ such that $F_y\big(x,f(x)\big)=0$. 
Thus by \eqref{eq:Rc=0}, $f^{(k)}(x) = c$ holds, except for at most finitely many values $x \in I$. Since $f$ is assumed to be $C^\infty$ this implies $f^{(k)}$ is identically a constant, so that $f$ is a polynomial of degree at most $k$, but this is a contradiction. Hence $F\nmid R_c$, and the proof is complete.
\end{proof}

We now apply Lemma~\ref{lem:5} to obtain our desired division into subintervals with control on the derivatives in each subinterval.
\begin{lemma}\label{lem:6}
Suppose $F(x,y)\in\R[x,y]$ is an irreducible polynomial of degree $d\ge2$. Let $I$ be an interval and $f(x)\in C^\infty(I)$ satisfy $F(x,f)=0$. Let $A_1,\ldots,A_k>0$. We may partition $I$ into $O(k^2d^2)$ many subintervals $I_\nu$ such that, for each interval $I_\nu$ and each $1\leq i\leq k$,
\begin{align*}
{\rm (i)}&\qquad |f^{(i)}(x)| \le A_i \qquad {\rm for \ all }\quad x\in I_\nu, \qquad {\rm or}\\
{\rm (ii)}&\qquad |f^{(i)}(x)| \ge A_i \qquad {\rm for \ all }\quad x\in I_\nu.
\end{align*}
\end{lemma}
\begin{proof}
By Lemma~\ref{lem:5} there are $O(kd^2)$ solutions to $f^{(i)}(x) = \pm A_i$ for each $1\leq i\leq k$. Let $x_1,\ldots,x_r$ be the union of all such solutions for $1\leq i\leq k$ (so $r=O(k^2d^2)$), ordered such that $x_0\leq x_1<\cdots<x_r\leq x_{r+1}$, writing $I=[x_0,x_{r+1}]$. By construction $f$ satisfies (i) or (ii) on each of the $r$ subintervals $I_\nu = [x_{\nu},x_{\nu+1}]$ for $\nu\leq  r$. The number of subintervals is $r=O(k^2d^2)$, as required.
\end{proof}

The previous lemma allows a division into subintervals where we can control the size of the derivatives. For an application of Lemma~\ref{lem-intbound1} we specifically require that the derivatives be \emph{small}, and hence will need a different method to handle the contribution from subintervals where the derivative remains large. Following Bombieri and Pila this is managed by the following lemma, which shows that such a subinterval must at least be very short, and then a trivial bound will suffice.

\begin{lemma}\label{lem:7}
Let $k\geq 1$, $0<\delta<1$, and $X>0$. Let $I$ be an interval and $f\in C^k(I)$. If
\begin{align*}
\Big|\frac{f^{(i)}(x)}{i!}\Big| \le X\delta^i
\quad\textrm{ for all }0\leq i<k
\quad\textrm{ while }\quad\Big|\frac{f^{(k)}(x)}{k!}\Big| \ge X\delta^k
\end{align*}
for every $x\in I$, then $\lvert I\rvert \leq 2/\delta$.
\end{lemma}
\begin{proof}
Write $I=[a,b]$. Considering the Taylor expansion at $x=a$, there exists $\xi\in I$ for which
\begin{align*}
f(b)-f(a) =\sum_{1\le i<k} \frac{f^{(i)}(a)}{i!}(b-a)^i + \frac{f^{(k)}(\xi)}{k!}(b-a)^k.
\end{align*}
Since $|f(b)|,|f(a)|\le X$ by assumption, isolating the remainder term gives
\begin{align*}
X\delta^k\lvert I\rvert^k \le \bigg|\frac{f^{(k)}(\xi)}{k!}(b-a)^k\bigg| \le 2X+ X\sum_{1\le i<k}\delta^i|I|^i.
\end{align*}
Denoting $\lambda = \delta \lvert I\rvert$, it follows that
\[\lambda^k \le 2+ \sum_{1\le i<k}\lambda^i= 2+\frac{\lambda^k-\lambda}{\lambda -1}.\]
We can assume $\lambda >1$ or we are done, and then we deduce $\lambda^k(\lambda-2)\leq \lambda -2$, whence $\lambda \leq 2$ as required.
\end{proof}

We can now use the previous lemmas to deduce a bound for the number of integral points on the graph of a sufficiently smooth function along a curve of bounded degree.
\begin{theorem}\label{thm:GNbound}
Let $d\geq 2$ and $N\geq 1$ be some integer sufficiently large in terms of $d$, and let $I$ be some interval of length $N$. Let $f\in C^\infty(I)$ be such that $\lvert f'(x)\rvert \leq 1$ for all $x\in I$. If there exists some irreducible polynomial $F(x,y)\in \R[x,y]$ of degree $d\geq 2$ such that $F(x,f(x))= 0$, identically as a function of $x$, then 
\begin{align*}
\left\lvert\big\{ (x,f(x)) : x\in I\big\}\cap \Z^2\right\rvert \le (\log N)^{O(d)}N^{\frac{1}{d}}.
\end{align*}
\end{theorem}
\begin{proof}
For any real $N>0$ define
\[
G(N):= \sup_{|I|\le N} \sup_{\substack{f\in C^\infty(I)\\ \lvert f'\rvert\leq 1}}\left\lvert \big\{(x,f(x)) : x \in I \big\}\cap \Z^2 \right\rvert.\]
We shall prove that for any integer $\ell \geq d$ there exists some $K=K(d,\ell)$ satisfying $2\leq K\leq \ell^{O(1)}$ such that for any $N>0$ we have
\begin{equation}\label{eq-recur}
G(N) \leq K^{O(d)} N^{\frac{1}{d}+O(\frac{1}{\ell})} + KG(K^{-2d}N).
\end{equation}
We first show how \eqref{eq-recur} implies Theorem \ref{thm:GNbound}. Fix some large $N>0$ and $\ell\geq d$ (which will depend on $N$). Iteratively applying \eqref{eq-recur}, since $K(K^{-2d})^{\frac{1}{d}}= K^{-1}\leq  1/2$ (assuming $\ell$ is large enough), it follows by induction that for any $n\geq 1$
\[G(N) \leq K^{O(d)}N^{\frac{1}{d}+O(\frac{1}{\ell})}\sum_{0\leq j<n}2^{-j}+K^nG(K^{-2nd}N).\]
In particular, if we choose $n$ large enough such that $K^{2nd}\in [N,K^{2d}N)$, we have $K^n \leq KN^{1/2d}$ and $G(K^{-2nd}N)\leq G(1)\le 1$, so that
\begin{align*}
G(N) & \ll K^{O(d)}N^{\frac{1}{d}+O(\frac{1}{\ell})}+KN^{1/2d} \ll \ell^{O(d)}N^{\frac{1}{d}+O(\frac{1}{\ell})} \ll(\log N)^{O(d)}N^{\frac{1}{d}},
\end{align*}
recalling $K\leq \ell^{O(1)}$ and choosing $\ell=\lceil \log N\rceil$. Thus \eqref{eq-recur} implies the result.

Now to prove \eqref{eq-recur} itself, we fix some interval $I$ of length $N$ and $f\in C^\infty(I)$ with $\lvert f'\rvert\leq 1$. Without loss of generality, translating the graph of $f$ by an integer if necessary, we may assume that $I=[0,N]$ and $\lvert f(x)\rvert \leq N$ for all $x\in I$. Fix some $\ell\geq d$ and let $\delta=\delta(d,\ell)\in(2/N,1)$ be some quantity to be chosen later. 

We shall apply Lemma \ref{lem:6} with $A_i=N\delta^i$, and  recall $D=|\mathcal M|\ll d\ell$ from \eqref{eq:Dcompute}. Indeed, by Lemma \ref{lem:6} we may partition $I$ into at most $O(d^2D^2)$ subintervals $I_\nu$ such that, for each $I_\nu$ and each $1\le i<D$, either 
\begin{align*}
{\rm (i)} & &\Big|\frac{f^{(i)}(x)}{i!}\Big| \ &\le N\delta^i \qquad {\rm for \ all }\quad x\in I_\nu,\textrm{ or }\\
{\rm (ii)} & &\Big|\frac{f^{(i)}(x)}{i!}\Big| \ &\ge N\delta^i  \qquad {\rm for \ all }\quad x\in I_\nu.
\end{align*}
Suppose first that $I_\nu$ satisfies $\lvert I_\nu\rvert \geq 1/\delta$ and (i) holds for every $1\leq i<D$. Then applying Lemma~\ref{lem-intbound1} we have
\begin{align*}
\left\lvert \big\{ (x,f(x)) : x\in I_\nu\big\}\cap \Z^2\right\rvert\ll
(d\ell)^2 N^{\frac{1}{d}+O(\frac{1}{\ell})}\delta|I_\nu|.
\end{align*}

Otherwise, either $\lvert I_\nu\rvert \leq 1/\delta$, or there exists a (minimal) $1\le k<D$ such that (ii) holds for $k$, but (i) holds for all $i<k$. In this case, Lemma \ref{lem:7} implies $|I_\nu| \le 2\delta^{-1}$.

Summing the contribution from each $I_\nu$ we deduce that 
\begin{align*}
\left\lvert \big\{(x,f(x)) : x \in I \big\}\cap \Z^2 \right\rvert & =  \sum_{\nu\ll d^2D^2}\left\lvert \big\{(x,f(x)) : x \in I_\nu \big\}\cap \Z^2 \right\rvert\\
& \ll (d\ell)^2 N^{\frac{1}{d}+O(\frac{1}{\ell})}\delta \sum_{\nu\ll d^2D^2}|I_\nu|+ d^2D^2G(2\delta^{-1}) \\
&\ll d^4\ell^2 \left( N^{\frac{1}{d}+O(\frac{1}{\ell})}\delta N +G(2\delta^{-1})\right),
\end{align*}
using the fact that $\sum_\nu \lvert I_\nu\rvert=\lvert I\rvert\leq N$. This gives \eqref{eq-recur} for some $K\ll d^4\ell^2\leq\ell^{O(1)}$, after choosing $\delta = 2K^{2d}/N$.
\end{proof}

We are now prepared to conclude the main result. This follows from the previous theorem and some elementary algebraic geometry, coupled with the inverse function theorem, to divide the curve into $O(1)$ many pieces with flat first derivatives. 

\begin{proof}[Proof of Theorem \ref{thm:bp}]
Let $C=\{(x,y)\in\R^2:F(x,y)=0\}$.  We first claim that $C\cap [0,N]^2$ has $O(d^2)$ many connected components. To show this, each connected component of $C\cap [0,N]^2$ forms either a loop, or a path from one boundary point of $[0,N]^2$ to another. 
In the case of a path, $C$ may intersect each of the boundary lines $[0,N]^2$ at $O(d)$ many points, by Corollary~\ref{cor:Bzout}, and hence $O(d^2)$ many such paths.
In the case of a loop, it must contain a point $(x,y)$ such that $F_x(x,y)=0$ or $F_y(x,y)=0$, where $F_x=\frac{\partial}{\partial x}F(x,y)$ and $F_y=\frac{\partial}{\partial y}F(x,y)$. And since both $F_x$ and $F_y$ are polynomials in $\mathbb{R}[x,y]$ of degree $\leq d-1$, by Corollary~\ref{cor:Bzout} there are at most $O(d^2)$ many points $(x,y)\in C$ such that $F_x(x,y)=0$ or $F_y(x,y)=0$. 
Hence in total there are $O(d^2)$ components. We fix one of these components, and call this $C_i$.

We now claim that there are $O(d^2)$ many points in $[0,N]^2$, which we remove, and $t=O(d^2)$ many open sets $U_1,\ldots,U_t$ which cover the remainder of $C_i$, such that
\[F_x(x,y)\neq 0\quad\textrm{ and }\quad F_y(x,y)\neq 0\qquad\textrm{ for all }\quad(x,y)\in \bigcup_i U_i,\]
and for all $1\leq i \leq t$ either (i) or (ii) holds:
\begin{align*}
{\rm (i)} & &\lvert F_x(x,y)\rvert \le \lvert F_y(x,y)\rvert\quad\textrm{ for all }\quad(x,y)\in U_i,\\
{\rm (ii)} & &\lvert F_y(x,y)\rvert \ge \lvert F_x(x,y)\rvert\quad\textrm{ for all }\quad(x,y)\in U_i.
\end{align*}
Indeed, as above there are $O(d^2)$ many points where $F_x=0$ or $F_y=0$, which we remove. If $F_x=\pm F_y$ then (i) and (ii) automatically hold. Otherwise, $F_x\mp F_y$ is a non-zero polynomial of degree $\leq d-1$, and hence again by Corollary~\ref{cor:Bzout} there are $O(d^2)$ many $(x,y)\in C_i$ such that $F_x(x,y)=\pm F_y(x,y)$. We can remove all such points by excising a further $O(d^2)$ many points, and then by continuity, either (i) or (ii) must hold along each segment of $C_i$ that remains after removing any of these $O(d^2)$ many points. The claim now follows, letting $U_i$ be some open narrow tube around each curve segment.

Now for each open $U_i\subseteq (0,N)^2$ as above, without loss of generality, (i) holds: $\lvert F_x\rvert\leq \lvert F_y\rvert$ on $U_i$. By the implicit function theorem \cite[Theorem 3.16]{Kn05}, there is an interval $I_i\subseteq [0,N]$ and a smooth function $f_i:I_i\to \mathbb{R}$ such that
\[C\cap U_i=\{(x,f_{i}(x)) : x\in I_i\}\]
and
\[f_{i}'(x) = -\frac{F_x(x,f_{i}(x))}{F_y(x,f_{i}(x))}\textrm{ for all }x\in I_i.\]
In particular $\abs{f_{i}'(x)}\leq 1$ for all $x\in I_i$ by (i). Hence by Theorem~\ref{thm:GNbound},
\begin{align}\label{eq:CUibound}
\lvert C\cap U_i\cap \{1,\ldots,N\}^2\rvert \ll (\log N)^{O(d)}N^{1/d}.
\end{align}
Since the open $U_i$ cover $C$, we conclude $\lvert C\cap \{1,\ldots, N\}^2\rvert \ll d^2 (\log N)^{O(d)}N^{1/d}.$
\end{proof}

\section{Integer points on convex graphs}\label{convex}
In this final section we give another application in Bombieri--Pila \cite{BP89}. This application highlights the versatility of the real-analytic determinant method; in particular this application is beyond the scope of the $p$-adic determinant method.

We say a function $f:\mathbb{R}\to \mathbb{R}$ is {\it strictly convex} if, between any two points on its graph, the line between those points lies strictly above the graph. That is, for any $x,y\in \mathbb{R}$ and $t\in (0,1)$ we have
\[f(tx+(1-t)y)<tf(x)+(1-t)f(y).\]
If $f$ is differentiable, this is equivalent to the derivative of $f$ being strictly increasing; if $f$ is twice differentiable, this is equivalent to $f''(x)>0$ pointwise.

It is natural question how many integer points lie on the graph of a strictly convex function, inside $[0,N]^2$ say. The example of $f(x)=x^2$ shows that $\gg N^{1/2}$ is possible. Some experimentation may suggest that this lower bound could be sharp, but this turns out to be false. As we shall prove, Jarnik \cite{Ja26} constructed a strictly convex function with $\gg N^{2/3}$ many points in $[0,N]^2$, which is best possible.

\begin{theorem}[Jarnik \cite{Ja26}]

If $f:[0,N]\to [0,N]$ is a strictly convex function, then
\[\lvert \{ (x,f(x)) : x\in [0,N]\}\cap \mathbb{Z}^2\rvert \ll N^{2/3}.\]
Moreover, for all large $N$ there exists a strictly convex function $f:[0,N]\to[0,N]$ such that
\[\lvert \{ (x,f(x)) : x\in [0,N]\}\cap \mathbb{Z}^2\rvert \gg N^{2/3}.\]
\end{theorem}
In particular, the implied constants are absolute and do not depend on $f$.
\begin{proof}
We first prove the upper bound. Suppose there are $0\leq n_1<\cdots <n_t\leq N$ such that $f(n_i)\in \mathbb{Z}$ for all $1\leq i\leq t$. Let $a_i=n_{i+1}-n_i$ and $b_i=f(n_{i+1})-f(n_i)$. Since $\sum_i a_i\leq N$ there are at least $\frac{3}{4}t$ many indices $i$ such that $a_i \leq 4N/t$, and similarly there are at least $\frac{3}{4}t$ many indices $i$ such that $b_i\leq 4N/t$. It follows there are at least $t/2$ many indices $i$ such that $\max(a_i,b_i)\leq 4N/t$. We note, however, that each $n_i$ gives rise to a distinct pair $(a_i,b_i)$, since by strict convexity we have, if $n_i<n_j$,
\[\frac{b_i}{a_i}=\frac{f(n_{i+1})-f(n_i)}{n_{i+1}-n_i}< \frac{f(n_{j+1})-f(n_j)}{n_{j+1}-n_j}=\frac{b_j}{a_j}.\]
Thus $t/2 \le \#\{(a_i,b_i) : \max(a_i,b_i)\leq 4N/t\} \le (4N/t)^2$, and hence $t\ll N^{2/3}$ as claimed.

We now construct the function for the lower bound. Let $H$ be some parameter to be chosen later. We consider all integer vectors $\mathbf{v}=(q,a)$ with $\mathrm{gcd}(a,q)=1$ and $1\leq a,q\leq H$, and order them as $\mathbf{v}_1,\ldots,\mathbf{v}_t$ such that $a_i/q_i<a_{i+1}/q_{i+1}$. Let $A_i=\sum_{1\leq j\leq i}a_j$ and $Q_i=\sum_{1\leq j\leq i}q_j$. Let $\tilde{f}:[0,Q_t]\to [0,A_t]$ be the piecewise linear function connecting the points $(Q_i,A_i)$ for $0\leq i\leq t$. Notice that 
\[Q_t=\sum_{1\le j\le t}q_j=\sum_{1\le q\le H}\sum_{\substack{1\leq a\leq H\\ (a,q)=1}}q \leq H^3,\]
(by symmetry, $A_t=Q_t\le H^3$) and the number of integer points on the graph of $\tilde{f}$ is
\[\geq t= \sum_{\substack{1\leq a,q\leq H\\ (a,q)=1}}1\gg H^2.\]
Since the gradient of the line segments is strictly increasing (the gradient between $(Q_i,A_i)$ and $(Q_{i+1},A_{i+1})$ is precisely $a_{i+1}/q_{i+1}$) the graph of $\tilde{f}$ is strictly convex if we consider only pairs of points on different line segments. We can make the entire graph strictly convex if we replace each line segment $\tilde f$ by a slight curve $f_\eps$. To make this explicit, one would compute $\tilde f(x)$ by linear interpolation through the endpoints $(Q_i, A_i)$ and $(Q_{i+1}, A_{i+1})$. Then for $\bar Q=\frac{Q_i + Q_{i+1}}{2}$, one defines $\tilde f_\eps (x)$ by Lagrange interpolation through the endpoints with the additional point $(\bar Q, \tilde f(\bar Q) + \eps)$. Each parabolic segment of $\tilde f_\eps$ now lies above its secant line, hence strictly convex, for $\eps=\eps_{N,\tilde f}>0$ sufficiently small.

This creates the graph of a strictly convex function $f:[0,H^3]\to [0,H^3]$ with $t\gg H^2$ many integer points $(Q_i,A_i)$. Setting $H=\lfloor N^{1/3}\rfloor$ completes the proof. 
\end{proof}

It may seem that this is the end of the story: we have lower and upper bounds of the same order of magnitude (and in fact Jarnik even gave refined bounds that match exactly up to lower order terms). Note, however, that Jarnik's construction of a strictly convex function with many integer points was piecemeal, and in particular was not smooth. Indeed, it is not even in $C^1$. One may hope, therefore, that the upper bound for the number of integral points can be improved granted additional smoothness hypotheses. 

Swinnerton-Dyer \cite{SD74} showed this is indeed true, proving an upper bound of $O_{f}(N^{3/5+o(1)})$ provided $f\in C^3$. A uniform (with no dependence on $f$) version of this result was proved by Schmidt \cite{Sc85}, under the additional assumption that $f^{(3)}\neq 0$ in $[0,N]$. Schmidt conjectured that the $3/5$ here could be improved to $1/2$ under the same assumptions.

In \cite{BP89} Bombieri and Pila gave, as another application of their determinant method, a proof that an upper bound of the strength $N^{1/2+o(1)}$ can be achieved provided $f$ is sufficiently smooth. As above, the example $f(x)=x^2$ shows that an exponent of $1/2$ is the best possible here. The proof is very similar to that of Theorem~\ref{thm:GNbound}, except that the assumption of strict convexity plays the role of B\'ezout's theorem, when bounding the number of integral points on the graph intersected with a line.
\begin{theorem}[Bombieri--Pila \cite{BP89}]
Let $d\geq 2$ and set $D=(d+1)(d+2)/2$. Let $f:[0,N]\to [0,N]$ be a strictly convex function. If $f\in C^D([0,N])$ and $f^{(D)}(x)\neq 0$ for all $x\in [0,N]$ then
\[ \lvert \{ (x,f(x)) : x \in [0,N]\}\cap \mathbb{Z}^2 \rvert \ll_d N^{\frac{1}{2}+\frac{8}{3(d+3)}+o(1)}.\]
\end{theorem}
In particular, the implied constant depends only on $d$ but not on $f$.
\begin{proof}
For any real $N>0$ define
\[
G(N):= \sup_{|I|\le N} \sup_{f}\left\lvert \big\{(x,f(x)) : x \in I \big\}\cap \Z^2 \right\rvert,\]
where the second supremum is over all strictly convex functions $f:I\to [0,N]$ such that $f\in C^D(I)$. We shall prove that there exists some $K=K(d)$ satisfying $2\leq K\leq d^{O(1)}$ such that for any $N>0$ we have
\begin{equation}\label{eq-recur2}
G(N) \leq K^{O(d)} N^{\frac{1}{2}+\frac{8}{3(d+3)}} + KG(K^{-2d}N).
\end{equation}
This implies the result by an identical argument to that in the proof of Theorem~\ref{thm:GNbound}. To prove \eqref{eq-recur2}, as in the proof of Theorem~\ref{thm:GNbound}, we let $\delta\geq 1/N$ be some parameter to be chosen later and want to divide $[0,N]$ into $d^{O(1)}$ many subintervals $I_\nu$ such that, for each $I_\nu$ and each $1\leq i<D$, either 
\begin{align*}
{\rm (i)} & &\Big|\frac{f^{(i)}(x)}{i!}\Big| \ &\le N\delta^i \qquad {\rm for \ all }\quad x\in I_\nu, \qquad\textrm{ or }\\
{\rm (ii)} & &\Big|\frac{f^{(i)}(x)}{i!}\Big| \ &\ge N\delta^i  \qquad {\rm for \ all }\quad x\in I_\nu.
\end{align*}
This time we do not have an assumption like $F(x,f)=0$ where $F$ has bounded degree, and so Lemma~\ref{lem:5} is not available. Instead, we note that the assumption that $f^{(D)}\neq 0$ directly implies $f^{(i)}(x)=c$ has at most $D-i$ solutions $x\in I$, for any $c\in \mathbb{R}$ and $1\leq i<D$. This can be used in place of Lemma~\ref{lem:5}, and hence such a subdivision can be found proceeding as in the proof of Lemma~\ref{lem:6}.

The length of an interval such that (ii) holds for some $1\leq i<D$ is, by Lemma~\ref{lem:7}, at most $2/\delta$. If (i) holds for all $1\leq i<D$, then by Lemma~\ref{lem:bp} we can cover the integer points from $I_\nu$ by 
\[\ll_d \delta \abs{I_\nu}N^{\frac{2dD}{3\binom{D}{2}}}+1=\delta \abs{I_\nu}N^{\frac{8}{3(d+3)}}+1\]
many integral curves of degree $\leq d$, where we choose $\mathcal{M}$ to be the set of all monomials of degree $\leq d$ (so that $D=\frac{1}{2}(d+1)(d+2)$ and $p=q=dD/3$). We now note that, by strict convexity, any line intersects the graph of $f(x)$ in at most 2 points. The number of integer points on any curve of degree $d\geq 2$ inside $[0,N]^2$ is, by Theorem~\ref{thm:bp}, at most $N^{1/2+o(1)}$. It follows that
\[G(N) \leq d^{O(1)}\delta \abs{I}N^{\frac{1}{2}+\frac{8}{3(d+3)}+o(1)}+G(2/\delta).\]
The conclusion now follows choosing $K=d^C$ and $\delta=K^{C'd}/N$ for some constants $C,C'$. 
\end{proof}

We close by restating the conjecture of Schmidt that asks for the same quality bound, while assuming a much weaker smoothness condition on $f$.
\begin{conjecture}[Schmidt \cite{Sc85}]
If $f:[0,N]\to [0,N]$ is a strictly convex function such that $f\in C^3([0,N])$ and $f^{(3)}(x)\neq 0$ for all $x\in [0,N]$ then
\[ \lvert \{ (x,f(x)) : x \in I\}\cap \mathbb{Z}^2 \rvert \ll N^{\frac{1}{2}+o(1)}.\]
\end{conjecture}
While still very much open, Schmidt's conjecture was stated as a central motivation for Bombieri and Pila's work \cite{BP89}. Using a strengthening of the real-analytic determinant method presented here,
Pila \cite{Pi91} has proved 
that such a bound holds if $f^{(105)}$ exists and does not vanish, as well as nonvanishing determinant of a $3\times 3$ matrix involving the first five derivatives of $f$---this may be viewed as an `enhanced convexity' condition. 

\bibliographystyle{amsplain}

\begin{thebibliography}{10}

\bibitem{BCN23} G. Binyamini, R. Cluckers, D. Novikov,
{\it Bounds for rational points on algebraic curves and dimension growth}, IMRN 11 (2024),
9256--9265.

\bibitem{Bir57} B. J. Birch, {\it Homogeneous forms of odd degree in a large number of variables}, Mathematika, 4 (1957), 102--105.

\bibitem{BP89} E. Bombieri, J. Pila, {\it The number of integral points on arcs and ovals}, Duke Math. J. 59 (1989), 337--357.

\bibitem{BoB20} D. Bonolis, T. Browning, {\it Uniform bounds for rational points on hyperelliptic fibrations}, Ann. Sc. Norm. Super. Pisa Cl. Sci. 24 (2021), 173-204. 

\bibitem{BoP22} D. Bonolis, L. Pierce, {\it Application of a polynomial sieve: beyond separation of variables},  Alg. Number Th. 18 (2024), 1515--1556.

\bibitem{BHBS06} T. D. Browning, D. R. Heath-Brown, P. Salberger, {\it Counting rational points on algebraic varieties}, Duke Math. J., 132 (2006), 545--578.

\bibitem{BLT24} T. Browning, J. D. Lichtman, J. Ter\"av\"ainen, {\it Bounds on the exceptional set in the abc conjecture}, preprint (2024) {\tt arXiv:2410.12234}


\bibitem{CCDN20} W. Castryck, R. Cluckers, P. Dittmann, K. H. Nguyen,
{\it The dimension growth conjecture, polynomial in the degree and without logarithmic factors}, Alg. Number Th. 14 (2020),  2261--2294.

\bibitem{CDHNV23} R. Cluckers, P. D\`ebes, Y. I. Hendel, K. H. Nguyen, F. Vermeulen,
{\it Improvements on dimension growth results and effective Hilbert's irreducibility theorem}, preprint (2023) {\tt arXiv:2311.16871}

\bibitem{Fal83} G. Faltings, {\it Endlichkeitss\"atze f\"ur abelsche Variet\"aten \"uber Zahlk\"orpern (Finiteness theorems for abelian varieties over number fields)}. Invent. Math. (in German), 73 (1983), 349--366.

\bibitem{GIP24} R. Greenfeld, M. Iliopoulou, S. Peluse, {\it On integer distance sets}, preprint (2024) {\tt arXiv:2401.10821}

\bibitem{Gr87} M. Gromov, {\it Entropy, homology, and semi-algebraic geometry}, Ast\'{e}risque 145-146 (1987), 225--240.

\bibitem{Hau17} Summer School, {\it Decoupling and Polynomial Methods in Analysis} Organizers: S. Guo, D. O. Silva, C. Thiele, Hausdorff Center for Mathematics, Bonn (2017)

{\tt https://www.math.uni-bonn.de/people/thiele/workshop19/SS17Kopp.pdf}

\bibitem{HB02} D. R. Heath-Brown, {\it The density of rational points on curves and surfaces}, Ann. of Math. (2) 155 (2002), 553--595.

\bibitem{Ja26} V. Jarnik, {\it \"{U}ber die Gitterpunkte auf konvexen Curven}, Math. Z. 24 (1926), 500–-518.

\bibitem{Kn05} A. W. Knapp, {\it Basic Real Analysis}, Birkh\"auser, Boston, 2005.

\bibitem{Ma10} O. Marmon, {\it A generalization of the Bombieri--Pila determinant method}, Zap. Nauchn. Sem. S.-Peterburg. Otdel. Mat. Inst. Steklov. 377 (2010), 63--77.

\bibitem{Pi91} J. Pila, {\it Geometric postulation of a smooth function and the number of rational points}, Duke Math. J. 63 (1991), 449--463.

\bibitem{Pi95} J. Pila, {\it Density of integral and rational points on varieties}, Ast\'{e}risque 228 (1995), 183--187.

\bibitem{Pi96} J. Pila, {\it Density of integer points on plane algebraic curves}, IMRN  18 (1996), 903--912.


\bibitem{PW06} J. Pila, A. Wilkie, {\it The rational points of a definable set}, Duke Math. J. 133 (2006), 591--616.


\bibitem{Sa07} P. Salberger, {\it On the density of rational and integral points on algebraic varieties}, J. Reine Angew Math. 606 (2007), 123--147.

\bibitem{Sa15} P. Salberger, {\it Uniform bounds for rational points on cubic 
hypersurfaces}, In: London Math. Soc. Lecture Note Ser. {\it Arithmetic and geometry} 420 (2007), 401--421

\bibitem{Sa23} P. Salberger, {\it Counting rational points on projective varieties}, Proc. London Math. Soc. 126 (2023), 1092--1133.

\bibitem{Sc85} W. M. Schmidt, {\it Integer points on curves and surfaces}, Monatsh. Math. 99 (1985), 45--72.

\bibitem{Se89} J. -P. Serre, {\it Lectures on the Mordell--Weil Theorem}, Aspects of Mathematics, E15, Friedr. Vieweg \& Sohn, Braunschweig, 1989, Translated from the French and edited by Martin Brown from notes by Michel Waldschmidt.

\bibitem{Se92} J. -P. Serre, {\it Topics in Galois theory}, Research Notes in Mathematics, vol. 1, Jones and Bartlett Publishers, Boston, MA, 1992, Lecture notes prepared by Henri Damon, With a foreword by Darmon and the author.

\bibitem{SD74} H. P. F. Swinnerton-Dyer, {\it The Number of Lattice Points on a Convex Curve}, J. Number Theory 6 (1974), 128--135.

\bibitem{Wa15} M. Walsh, {\it Bounded rational points on curves}, IMRN 14 (2015), 5644--5658.

\bibitem{Yo87} Y. Yomdin, {\it $C^k$-resolution of semi-algebraic mappings}, Addendum to: Volume growth and entropy, Israel J. Math. 57 (1987), 301--317.

\end{thebibliography}

\end{document}